\documentclass{article}
\usepackage{krstruct}
\newtheorem{thm}{Theorem}[section]

\newtheorem{lem}[thm]{Lemma}
\newtheorem{cor}[thm]{Corollary}
\newtheorem{conj}[thm]{Conjecture}
\newtheorem{prop}[thm]{Proposition}

\newtheorem{quest}[thm]{Question}

\newtheorem{example}[thm]{Example}

\newtheorem{remark}[thm]{Remark}

\makeatletter
\renewenvironment{proof}[1][\proofname]{%
   \par\pushQED{\qed}\normalfont%
   \topsep6\p@\@plus6\p@\relax
   \trivlist\item[\hskip\labelsep\bfseries#1\@addpunct{.}]%
   \ignorespaces
}{%
   \popQED\endtrivlist\@endpefalse
}

\usepackage{float}
\usepackage[caption = false]{subfig}

\title{Totally geodesic surfaces in twist knot complements}
\author{Khanh Le and Rebekah Palmer}
\date{}

\begin{document}

\large

\maketitle

\begin{abstract}
In this article, we give explicit examples of infinitely many non-commensurable (non-arithmetic) hyperbolic $3$-manifolds admitting exactly $k$ totally geodesic surfaces for any positive integer $k$, answering a question of Bader, Fisher, Miller and Stover. The construction comes from a family of twist knot complements and their dihedral covers. The case $k=1$ arises from the uniqueness of an immersed totally geodesic thrice-punctured sphere, answering a question of Reid. Applying the proof techniques of the main result, we explicitly construct non-elementary maximal Fuchsian subgroups of infinite covolume within twist knot groups, and we also show that no twist knot complement with odd prime half twists is right-angled in the sense of Champanerkar, Kofman, and Purcell.
\end{abstract} 



\section{Introduction}
\label{sec:Intro}

The study of surfaces has been essential in studying the geometry and topology of the $3$-manifolds that contain them. In this paper, we will mainly be concerned with complete properly immersed totally geodesic surfaces in hyperbolic $3$-manifolds. There has been considerable work in understanding the existence of totally geodesic surfaces in hyperbolic $3$-manifolds. Menasco and Reid \cite[Corollary 4]{MenascoReid} proved that the complement of a hyperbolic tunnel number one knot cannot contain a closed embedded totally geodesic surface. Calegari \cite[Corollary 4.6]{C} showed that a fibered knot complement in a rational homology sphere whose trace field has odd prime degree cannot contain an immersed totally geodesic surface. By Adams and Schoenfeld \cite[Theorem 4.1]{A05}, two-bridge knot and link complements do not contain any embedded orientable totally geodesic surfaces. On the other hand, there are examples of hyperbolic $3$-manifolds that do contain totally geodesic surfaces. Adams \cite[Theorem 3.1]{A85} proved that any incompressible thrice-punctured sphere in a hyperbolic $3$-manifold is totally geodesic. Adams and Schoenfeld \cite[Example 3.1]{A05} exhibited examples of balanced pretzel knot containing a totally geodesic Seifert surface. For arithmetic hyperbolic $3$-manifolds, it is known that if there exists at least one totally geodesic surface, then there are in fact infinitely many totally geodesic surfaces. 

Most recently, Fisher, Lafont, Miller and Stover provided the first examples of hyperbolic $3$-manifolds whose set of totally geodesic surfaces is nonempty and finite \cite{FLMS}. They proved finiteness of maximal totally geodesic submanifolds of dimension at least 2 for a large class of non-arithmetic hyperbolic $n$-manifolds \cite[Theorem 1.2]{FLMS}. In a subsequent work, Bader, Fisher, Miller and Stover showed that if a complete finite-volume hyperbolic $n$-manifold of dimension at least 3 contains infinitely many maximal totally geodesic submanifolds then it must be arithmetic \cite[Theorem 1.1]{BFMS}. A similar result was also obtained for the case of closed hyperbolic $3$-manifolds by Margulis and Mohammadi \cite[Theorem 1.1]{MM}. Bader, Fisher, Miller and Stover \cite[Question 5.4]{BFMS} asked the following natural question, which is the motivation of our paper: 

\begin{quest}
\label{ques:BFMS}
For each $k \geq 1$, is there a hyperbolic $3$-manifold containing exactly $k$ totally geodesic surfaces?
\end{quest}

We answer yes to the above question by exhibiting explicit examples of hyperbolic $3$-manifolds with exactly $k$ totally geodesic surfaces:

\begin{thm}
\label{thm:kTGSurfaces}
Let $k$ be any positive integer. There exist infinitely many non-commensurable hyperbolic $3$-manifolds with exactly $k$ totally geodesic surfaces.
\end{thm}

These are the first examples of hyperbolic $3$-manifolds for which the number of totally geodesic surfaces is positive, finite, and precisely known. The idea behind \cref{thm:kTGSurfaces} is to find hyperbolic $3$-manifolds with exactly one totally geodesic surface and then consider particular covers of these manifolds. 

We specifically study twist knot complements. For each positive integer $j$, let $K_j$ be the twist knot with exactly $j$ half-twists as shown in \cref{fig:twistknot}. Let $M_j := S^3 \smallsetminus K_j$. These twist knots admit a complete finite-volume hyperbolic metric if and only if $j \geq 2$. In other words for $j\geq 2$, the knot group $\Gamma_j := \pi_1(S^3 \smallsetminus K_j)$ admits a discrete faithful representation $\rho_j$ into $\PSL_2(\bbC)$. By Adams \cite[Theorem 3.1]{A85}, the thrice-punctured sphere $N$, seen as the dotted surface in \cref{fig:twistknot}, is totally geodesic for $j \geq 2$.  For simplification of notation, we will use $N$ to denote this surface, independent of $j$.

\begin{figure}[h!]
    \centering
    \includegraphics[height=1.8in]{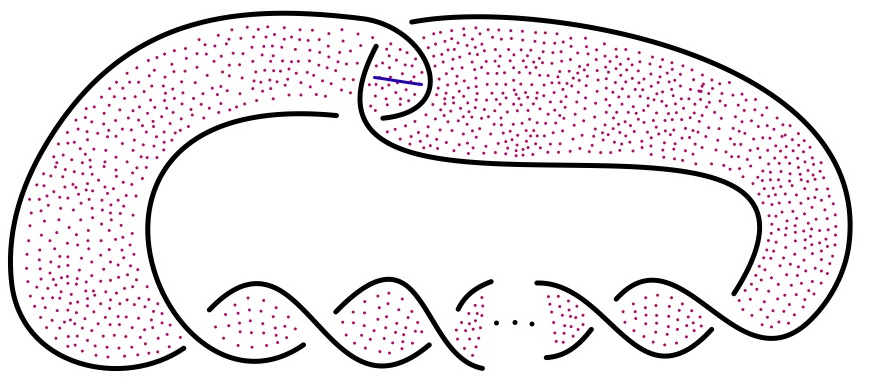}
    \caption{Twist knot $K_j$ and the thrice-punctured sphere $N$}
    \label{fig:twistknot}
\end{figure}

By Reid \cite[Theorem 3]{Reid}, among the twist knots with $j\geq 3$ (all but the figure-8 knot, $j=2$ in our notation), there are infinitely many $j$ so that $M_j$ contains no closed totally geodesic surfaces and has exactly one commensurability class of cusped totally geodesic surfaces. In fact, Reid conjectured that the thrice-punctured sphere is the unique totally geodesic surface inside $M_j$ for $j\geq 3$. The following theorem of our paper confirms this conjecture for infinitely many values of $j$.

\begin{thm}
\label{thm:UniqueThricePuncturedSphere}
Let $M_j$ be the complement of a twist knot whose trace field has odd degree over $\bbQ$ and contains no proper real subfield besides $\bbQ$. Then the thrice-punctured sphere $N$ is the unique totally geodesic surface in $M_j$.   
\end{thm}

Since $M_j$ can be obtained by doing Dehn filling on the Whitehead link, the volume of $M_j$ is uniformly bounded above by that of the Whitehead link. We note that the condition on the trace field in \cref{thm:UniqueThricePuncturedSphere} is the same as the one considered by Reid in \cite[Proposition 2]{Reid}. As noted in \cite[Lemma 5]{Reid}, it was originally due to Riley \cite{Riley} that for infinitely many primes $j$, the trace field of $M_j$ satisfies the condition in \cref{thm:UniqueThricePuncturedSphere}. Later work of Hoste and Shanahan \cite{HS} implies the condition in \cref{thm:UniqueThricePuncturedSphere} holds for the trace field of $M_j$ for all prime $j$. Using Magma \cite{Magma} and SageMath \cite{sage}, we also verify that the condition in \cref{thm:UniqueThricePuncturedSphere} holds for the trace field of $M_j$ for $j$ odd and $j \leq 99$, see \cite{Codes} for the code. Therefore, we obtain the following:

\begin{cor}
\label{cor:ExistUnique}
Suppose $j$ is an odd prime or $j>1$ is an odd number less than or equal to $99$. The thrice-punctured sphere $N$ is the unique totally geodesic surface in $M_j$.
\end{cor}

The analysis in the proof of \cref{thm:UniqueThricePuncturedSphere} has several interesting consequences. For all odd primes $j$, we give explicit examples of non-elementary maximal Fuchsian subgroup of infinite covolume in $\Gamma_j$. By Maclachlan and Reid \cite[Theorem 4]{MaclachlanReid87}, every non-elementary maximal Fuchsian subgroup of an arithmetic Kleinian group has finite covolume. Therefore, \cref{cor:ExistUnique} illustrates a stark difference between arithmetic and non-arithmetic $3$-manifolds. Infinite covolume Fuchsian groups of this kind are known to exist, for example using results in \cite{FLMS}, but we do not know of any explicit examples in the literature.       

Another consequence of our analysis relates to right-angled knots. A knot in $S^3$ whose complement admits a decomposition into ideal hyperbolic right-angled polyhedra is called a \emph{right-angled knot}. Such a decomposition gives immersed totally geodesic surfaces coming from the faces of the polyhedra which meet at right angles. For $j \geq 3$, our analysis of the geometry of $N$ in $M_j$ gives a strong restriction on the boundary slopes of cusped totally geodesic surfaces. This allows us to show that the angle at an intersection of cusped totally geodesic surfaces in $M_j$ is never a right angle.

\begin{cor}
\label{cor:NoRightAngled}
The twist knot $K_j$ is not right-angled for $j$ odd prime.  
\end{cor}

\cref{cor:NoRightAngled} confirms for infinitely many knots the following conjecture by Champanerkar, Kofman and Purcell \cite[Conjecture 5.12]{CKP}:

\begin{conj}[Champanerkar, Kofman, and Purcell]
\label{conj:NoRightAngledKnot}
There does not exist a right-angled knot.
\end{conj}

Previously known evidence for this conjecture comes from knot complements with no totally geodesic surface, such as knot complements satisfying a condition described by Calegari \cite[Corollary 4.6]{C}. Since the knot $8_{20}$ satisfies \cite[Corollary 4.6]{C}, it does not contain totally geodesic surface and is not right-angled. Examples of knot complements with totally geodesic surfaces supporting \cref{conj:NoRightAngledKnot} have been found among knots with small crossing numbers. In particular, the conjecture was verified for knots with up to 11 crossings by Champanerkar, Kofman, and Purcell \cite{CKP}. As far as we know, \cref{cor:NoRightAngled} provides the first infinitely family of knots that contain at least one totally geodesic surface and are not right-angled.

Finally, our techniques allow us to investigate boundary slopes of totally geodesic surfaces in the figure-8 knot. We prove that any rational number occurs as a boundary slope of a totally geodesic surface in the figure-8 knot in \cref{thm:BoundSlopeQ}.

Our article is organized as follows. In \cref{sec:Prelim}, we collect some facts about twist knot complements. In \cref{sec:MainThms}, we prove \cref{thm:UniqueThricePuncturedSphere} and \cref{thm:kTGSurfaces}. At the end of \cref{sec:MainThms}, we establish \cref{cor:NoRightAngled}. In \cref{sec:Fig8Knot}, we study boundary slopes of totally geodesic surfaces in the figure-8 knot. Finally in \cref{sec:ComputerExperimentsAndQuestions}, we discuss some computer experiments and some open questions.  

\subsection*{Acknowledgements} We would like to thank our advisor Matthew Stover for asking us \cref{ques:BFMS}, for multiple helpful suggestions in proving the main theorems, and for the careful reading of the first draft of this paper. We also thank David Futer and Rose Kaplan-Kelly for pointing out the application of \cref{lem:SlopesOfTGS} towards right-angled knots. Finally, we would like to thank the referee for their comments and their suggestions to use Sage, an open source,  and the initial Sage codes. 


\section{Preliminaries}
\label{sec:Prelim}
\subsection{Twist knots complements}

Let $K_j$ be the twist knot with $j$ half-twists, $M_j:=S^3 \smallsetminus K_j$, and $\Gamma_j$ be its fundamental group. The group $\Gamma_j$ has a presentation $\Gamma_j = \langle a,b \mid b = w_j a w_j^{-1} \rangle$ where $w_j$ is given by a word of length $2j$ in $a^{\pm 1}$ and $b^{\pm 1}$, namely:
\begin{equation}
\label{eq:wj}
w_j  = 
\begin{cases}
(a^{-1}b^{-1} ab)^{(j-1)/2}  a^{-1}b^{-1} & j \equiv 1 \mod 2 \\
(a^{-1}b ab^{-1})^{j/2}  & j \equiv 0 \mod 2
\end{cases}
\end{equation}
The two generators $a$, $b$ correspond to the two meridians of the twist knot chosen to have the same orientation. The longitude of the twist knots that commute with $a$ can be computed diagrammatically and is given by:
\begin{equation}
\label{eq:longitude}
\ell_j = 
\begin{cases}
a \overline{w}_j b^2 w_j a & j \equiv 1 \mod 2 \\
\overline{w}_j w_j  & j \equiv 0 \mod 2 \\
\end{cases}
\end{equation}
where $\overline{w}_j$ is $w_j$ spelled backwards. Finally, we note that the above presentation is different from that in \cite[Equation 2.1]{HS} by an isomorphism interchanging $a$ and $b$.  

%

For $j\geq 2$, $M_j$ admits a complete hyperbolic structure of finite volume. The discrete faithful representation $\rho_j:\Gamma_j \to \SL_2(\bbC)$ sends all conjugates of the meridians of the knot to parabolic isometries; i.e.~conjugate in $\SL_2(\bbC)$ to: 
 
\[\begin{pmatrix} 1 & 1 \\ 0 & 1 \end{pmatrix}\]

\begin{thm}
\label{thm:rho_j}
Suppose we normalize so that:
\begin{equation*}
    \rho_j(a) = \begin{pmatrix}
    1 & 1 \\ 0 & 1
    \end{pmatrix} \quad 
    \text{and} \quad
    \rho_j(b) = \begin{pmatrix}
    1 & 0 \\ z_j & 1
    \end{pmatrix}
\end{equation*}
Then $\rho_j$ defines the discrete faithful representation if and only if $z_j$ satisfies a polynomial $\Lambda_j(z) \in \bbZ[z]$. The polynomial $\Lambda_j(z)$ is defined recursively by
\begin{equation}
    \label{eq:min_poly(z)}
    \Lambda_{j+2}(z) =(z^2+2)\Lambda_j(z) - \Lambda_{j-2}(z)
\end{equation} 	
where $\Lambda_{-1}(z) = \Lambda_0(z) = 1$, $\Lambda_1(z) = z+ 1 $, and $\Lambda_2(z) = z^2 - z + 1$.  
\end{thm}

The fact that $z_j$ must be a root of an integral polynomial was first observed by Riley for the class of two-bridge knots \cite[Theorem 2]{Riley}. For the twist knots, a recursion for the polynomials $\Lambda_j(z)$ was given by Hoste and Shanahan \cite[Theorem 2]{HS}. We make the following observation which will be used in the proof of \cref{thm:UniqueThricePuncturedSphere}:

\begin{remark}
\label{rem:zjPowerBasis}
The image of $\Gamma_j$ under $\rho_j$ is contained in $\SL_2(\bbZ[z_j])$. In particular, the entries of the matrices in $\rho_j(\Gamma_j)$ can be written as $\bbZ$-linear combinations of elements in the $z_j$-power basis.
\end{remark}

Since the presentation of $\Gamma_j$ that we are using is slightly different from that in \cite[Equation 2.1]{HS}, we give a proof of the above theorem for completeness.    


\begin{proof}[Proof of \cref{thm:rho_j}] Consider the free group $\F_2$ on two generators $A$ and $B$ along with the surjective homomorphism $\F_2 \to \Gamma_j$ sending $A$, $B$ to $a$, $b$ respectively. Let $W_j$ be the lift of $w_j$ given by \cref{eq:wj}. Consider the homomorphism $P: \F_2 \to \SL_2(\bbZ[z])$ defined by 
\[P(A) = \begin{pmatrix}
1 & 1 \\ 0 & 1 \\
\end{pmatrix}
\quad \text{and} \quad
P(B) = \begin{pmatrix} 1 & 0 \\ z & 1 \end{pmatrix}\]
We claim that
\begin{equation}
\label{eq:Matrix_Wj}
P(W_j) = 
\begin{pmatrix}
\Lambda_j(z) & \mu_j(z) \\ z \mu_j(z) & \Lambda_{j-1}(z) 
\end{pmatrix}
\end{equation}
for some $\mu_j(z) \in\bbZ[z]$.

For convenience of notation, we identify elements of $\F_2$ with their image under $P$. We consider two cases according to whether $j$ is even or odd. When $j$ is even, we have $W_j = ([A^{-1},B])^{j/2}$. Applying Cayley--Hamilton to $[A^{-1},B]$, we get 
\[
	([A^{-1},B])^2 - \tr([A^{-1},B]) [A^{-1},B] + I = 0
\]
Note that $\displaystyle \tr([A^{-1},B]) = z^2 + 2$. Therefore, we have 
\begin{equation}
\label{eq:W_j_Recur}
	W_{j+2} = (z^2+2) W_{j} - W_{j-2}
\end{equation}
This gives us the recursion in \cref{eq:min_poly(z)}. Now we observe that the upper diagonal (resp.~lower diagonal) entries of $W_2= [A^{-1},B]$ and of $W_0 = I$ give us the initial conditions for $\Lambda_j(z)$ (resp.~$\Lambda_{j-1}(z)$). For the off-diagonal entries, we observe that they satisfy the same recursion as in \cref{eq:min_poly(z)} and their initial conditions differ by a factor of $z$. When $j$ is odd, we proceed similarly to compute $P(W_j)$ where the initial conditions are given by $W_{-1} = B^{-1}A^{-1}$ and $W_1 = A^{-1}B^{-1}$.

To finish the proof, we note that $\rho_j$ is obtained by factoring $\ev_j\circ P$ through the canonical projection $\F_2 \to \Gamma_j$ where $\ev_j:\SL_2(\bbZ[z])\to \SL_2(\bbZ[z_j])$ is the homomorphism induced by the evaluation map $\bbZ[z] \to \bbZ[z_j]$. Using \cref{eq:Matrix_Wj}, we see that the relation $w_j a - b w_j =0$ is satisfied if and only if $\Lambda_j(z_j) = 0$.
\end{proof}

Given a finitely-generated non-elementary subgroup $\Gamma$ of $\PSL_2(\bbC)$, we can define the \emph{trace field} of $\Gamma$ to be $\bbQ(\tr \Gamma) := \bbQ(\tr \gamma \mid \gamma \in \Gamma)$ and the \emph{invariant trace field} of $\Gamma$ to be $k\Gamma := \bbQ(\tr \gamma^2 \mid \gamma \in \Gamma)$. If $\Gamma$ additionally has finite covolume, then $\bbQ(\tr \Gamma)$ and $k\Gamma$ are number fields \cite[Theorem 3.1.2]{MaclachlanReid}. In general, the trace field and invariant trace field are different. The trace field is not an invariant of the commensurability class of $\Gamma$, but the invariant trace field is. These fields may coincide, such as link complements in a $\bbZ/2$-homology sphere. \cite[Corollary 4.2.2]{MaclachlanReid}. Finally, suppose that $\Gamma$ is generated by $\gamma_1$ and $\gamma_2$; then $\bbQ(\tr \Gamma) = \bbQ(\tr \gamma_1, \tr \gamma_2,  \tr \gamma_1\gamma_2)$ \cite[Equation 3.25]{MaclachlanReid}.

It follows from the discussion above and \cref{thm:rho_j} that the trace field of twist knot complements is $\bbQ(z_j)$. Hoste and Shanahan proved that $\Lambda_j(z)$ is the monic minimal polynomial of $z_j$ over $\bbZ$ of degree $j$ \cite[Theorem 1]{HS}. Therefore, $[\bbQ(z_j):\bbQ] = j$ and so twist knots are pairwise non-commensurable.

\begin{cor}
\label{cor:Matrixwj}
Under the discrete faithful representation $\rho_j$, we have
\begin{equation*}
\rho_j(w_j) = 
	\begin{pmatrix}
	0 & \mu_j \\ z_j \mu_j & [(-1)^{j}z_j-2]\mu_j \\
	\end{pmatrix}
	\quad \text{and} \quad 
\rho_j(\overline{w}_j) = 
	\begin{pmatrix}
	 [(-1)^{j}z_j-2]\mu_j & \mu_j \\ z_j \mu_j & 0  \\
	\end{pmatrix}
\end{equation*}
where $\mu_j = \mu_j(z_j)$. Furthermore, the image of $\ell_j$ is given by: 
\begin{equation*}
\label{eq:longitude-rho}
\rho_j(\ell_j) = 
\begin{pmatrix}
-1 & -\tau_j \\ 0 & -1
\end{pmatrix}
\end{equation*}
where $\tau_j= 4\mu_j^2 + 2$.
\end{cor}

\begin{proof}
Let us consider $\overline{W}_j$ to be the lift of $\overline{w}_j$ in $\F_2$ given by spelling $W_j$ backwards. We first claim that:
\[
P(\overline{W}_j) =  \begin{pmatrix}
\Lambda_{j-1}(z) & \mu_j(z) \\ z \mu_j(z) & \Lambda_{j}(z)
\end{pmatrix} 
\]
Observe that $\overline{W}_j$ satisfy the same recurrence in \cref{eq:W_j_Recur}, where $\overline{W}_0$, $\overline{W}_1$ provide the base case. 

We next claim that: 
\[ (-1)^{j+1}\Lambda_{j}(z) + \Lambda_{j-1}(z) = [(-1)^j z - 2] \mu_j(z)\]
We proceed by induction on $j$. The base case is provided by $W_0 = I$ and $W_1 = A^{-1}B^{-1}$. Assume the statement for $W_k$ for all $k< j+2$. We have 

\begin{align*}
(-1)^{j+3}\Lambda_{j+2} (z)+ \Lambda_{j+1}(z) 
	&= (z^2 +2) [(-1)^{j+3}\Lambda_j(z) + \Lambda_{j-1}(z)])  - [(-1)^{j+3}\Lambda_{j-2}(z) + \Lambda_{j-3}(z)] \\
	&=(z^2 + 2)[(-1)^jz-2]\mu_j(z) - [(-1)^j z-2]\mu_{j-2}(z)\\
	&= [(-1)^{j+2}z-2]\mu_{j+2}(z) 
\end{align*}
where the last equality follows from the recurrence of $\mu_j(z)$ given by \cref{eq:W_j_Recur}. Therefore, we have
    \[ \Lambda_{j-1}(z_j) = [(-1)^jz_j-2]\mu_j \]
This identity gives us the image of $w_j$ and $\overline{w}_j$ under $\rho_j$. We note that by the determinant of $P(\overline{W}_j)$,
\begin{equation}
\label{eq:zj_muj}
    z_j \mu_j^2 = -1
\end{equation}

Applying $\rho_j$ to \cref{eq:longitude} and using \cref{eq:zj_muj}, we get the image of $\ell_j$ under $\rho_j$ where 
\begin{equation}
\label{eq:tau}
\tau_j = (4\mu_j^2 + 2)=-\frac{4}{z_j} + 2
\end{equation}
\end{proof}

\begin{remark}
Observe that since $\mu_j \in \bbZ[z_j]$, \cref{eq:zj_muj} implies that $z_j$ and $\mu_j$ are units in $\bbZ[z_j]$. We also note that Equations \eqref{eq:zj_muj} and \eqref{eq:tau} are used extensively in the proof of \cref{thm:UniqueThricePuncturedSphere}. 
\end{remark}

We end this subsection by defining boundary slopes for cusped totally geodesic surfaces in $M_j$. Let us identify the universal cover $\bbH^3$ of $M_j$ with the upper-half space model and the ideal boundary of $\bbH^3$ with $\bbC \cup \{\infty\}$. The action of $\Gamma_j$ on $\bbH^3$ is given by $\rho_j$. We say that $P\in \bbC \cup \{\infty\}$ is a \emph{cusp point} of $M_j$ if its stabilizer in $\Gamma_j$ is conjugate to $\langle a ,\ell_j \rangle$. In particular, the point $\infty$ is the cusp point of $M_j$ whose stabilizer in $\Gamma_j$ is the $\bbZ^2$-subgroup generated by $a$ and $\ell_j$ acting as translations by $1$ and $\tau_j$ respectively. At every cusp point in $\bbH^3$, we choose a sufficiently small horoball neighborhood such that the stabilizer is a conjugate of $\langle a, \ell_j\rangle$ and such that these horoball neighborhoods are invariant under the action of $\Gamma_j$. We shall refer to the image of these horoball neighborhood as the \emph{cusp} of $M_j$. Similarly given any hyperplane $H$ in $\bbH^3$, a point $P$ in the circle/line at infinity of $H$ is called a cusp point of $H$ if it is a fixed point of a parabolic element in $\Stab_{\Gamma_j}(H)$.

Let $\Sigma$ be a properly immersed cusped totally geodesic surface in $M_j$. Since $M_j$ has one cusp and $\Sigma$ is complete, all cusps of $\Sigma$ must be contained in the cusp of $M$. Given any cusp $c$ of $\Sigma$, we consider a lift $\widetilde{\Sigma}$ of $\Sigma$ to $\bbH^3$ by putting the cusp point corresponding to $c$ at $\infty$. The cusp cross-section of $c$ in $\Sigma$ is obtained by intersecting $\Sigma$ with the cusp cross-section of $M$. This intersection lifts to a horocycle at $\infty$ in $\widetilde{\Sigma}$. The stabilizer of this horocycle is an element of the form $a^p \ell_j^q$ in $\langle a,\ell_j  \rangle$. In this case, we say that $p/q$ is a \emph{boundary slope} of $\Sigma$. 

More generally, suppose that $\zeta$ is a cusp point of a hyperplane $H$ in $\bbH^3$ such that the parabolic element in the stabilizer of $H$ in $\Gamma_j$ fixing $\zeta$ is conjugate to $a^p\ell_j^q$ in $\Gamma_j$. Then we say that $p/q$ is the boundary slope of $H$ at $\zeta$.  

\subsection{The thrice-punctured sphere $N$}

For all $j \geq 2$,  $M_j$ contains an immersed thrice-punctured sphere $N$ as shown in \cref{fig:twistknot}. By Adams \cite[Theorem 3.1]{A85}, $N$ is a totally geodesic surface. Moreover:

\begin{prop}
\label{prop:ThricePuncturedSphere}
For all $j \geq 2$, the immersed surface $N$ is isotopic to a totally geodesic thrice-punctured sphere in $M_j$ whose fundamental group is generated by $x:=a$ and 

\begin{equation*}
	y_j := 
	\begin{cases}
	a w_j a b a^{-1} w_j^{-1} a^{-1} &\text{if } j \equiv 1 \mod 2\\
	a w_j a^{-1} b a w_j^{-1} a^{-1} &\text{if } j \equiv 0 \mod 2\\
	\end{cases}
\end{equation*}
The set of boundary slopes of $N$ is $\{1/0,1/0,-2/1\}$.
\end{prop}

\begin{proof}
From the diagram, we see that the fundamental group of $N$ is generated by $b$ and:
\begin{equation*}
\begin{cases}
a^{-1}(bab^{-1}a^{-1})^{(j-1)/2} bab^{-1}(aba^{-1}b^{-1})^{(j-1)/2}a
& j \equiv 1 \mod 2
\\
(b^{-1}aba^{-1})^{j}b^{-1}ab(ab^{-1}a^{-1}b)^{j}
& j \equiv 0 \mod 2
\end{cases}
\end{equation*}
Thus, we see that $\pi_1(N)$ is conjugate to the subgroup $\Delta_j$ generated by $x:=a$ and $y_j$ as stated. Under the discrete faithful representation, we have
\begin{equation*}
	\rho_j(\Delta_j) = \left\langle \begin{pmatrix}
	1 & 1 \\ 0 & 1
	\end{pmatrix}, 
	\begin{pmatrix}
	-1 & 1 \\ -4 & 3 \\
	\end{pmatrix}
	\right \rangle
\end{equation*}
which is conjugate to the principal congruence subgroup of level 2 of $\PSL_2(\bbZ)$.

Note that since both $x$ and $y_j$ are conjugate to $a$, the surface $N$ has two slopes 1/0. The parabolic element corresponding to the remaining cusp is given by 
\begin{equation*}
	x^{-1} y_j^{-1} = 
	\begin{pmatrix}
	-1 & 0 \\ 4 & -1 
	\end{pmatrix} = b^{-2} w_j \ell_j w_j^{-1}
\end{equation*}  
Since $b^{-2} w_j \ell_j$ and $a^{-2} \ell_j$ are conjugate to each other in $\Gamma_j$, the slope of the remaining cusp is $-2/1$. 
\end{proof}	

\begin{remark}
\label{rmk:ElementsOfDelta}
Let $\bbH^2_\bbR$ be the hyperplane in $\bbH^3$ such that the line at infinity is the real line $\bbR$. The group $\rho_j(\Delta_j)$ stabilizes $\bbH^2_\bbR$ in $\bbH^3$. In fact, it is conjugate to the principal congruence subgroup of level 2 of $\PSL_2(\bbZ)$ by the matrix
\[ \begin{pmatrix} 1/\sqrt{2} & 0 \\ 0 & \sqrt{2} \end{pmatrix}\]
Thus, elements of $\rho_j(\Delta_j)$ are precisely matrices in $\PSL_2(\bbZ)$ such that the lower diagonal entry is congruent to $0 \mod 4$.
This group is also known as the Hecke congruence subgroup of level $4$.
\end{remark}


We will also need the following fact about the action of $\rho_j(\Delta_j)$ on $\bbH_\bbR^2$. 

\begin{lem}
\label{lem:CuspsOrbits}
Consider the action of $\rho_j(\Delta_j)$ on $\bbH^2_\bbR$. The cusp points of this action is $\bbQ \cup \{\infty\}$, which is partitioned into 3 distinct orbits, namely
\begin{align*}
	[0] &= \{u/v \in \bbQ \mid v \equiv 1 \mod 2\}, \\
	[\tfrac{1}{2}] &= \{u/v \in \bbQ \mid v \equiv 2 \mod 4 \}, \\
	[\infty] &= \{ u/v \in \bbQ \mid v \equiv 0 \mod 4\} \cup \{\infty\} 
\end{align*} 
where $u,v \in \bbZ$ are relatively prime.	
\end{lem}

\begin{proof}
Let $C$ denote the set of cusp points in $\bbH^2_\bbR$ with respect to the action of $\rho_j(\Delta_j)$. Since $a \in \rho_j(\Delta_j)$, we see that $\infty \in C$. The fundamental domain for the action is an ideal quadrilateral with vertices at $\infty, 0, \frac{1}{2}$ and $1$. There are three distinct orbits corresponding to $0$, $\frac{1}{2}$ and $\infty$. To see the description of these orbits, we first note that the integer matrices
\begin{equation*}
	\begin{pmatrix}
	v' & u \\ -4u' & v 
	\end{pmatrix}, \quad 
	\begin{pmatrix}
	u -2v' & v' \\ v-2u'  & u' 
	\end{pmatrix}, \text{ and} 
	\begin{pmatrix}
	u & v' \\ v & u' 
	\end{pmatrix}
\end{equation*}
map $0$, $\frac{1}{2}$, and $\infty$ to $u/v$ when $v$ is odd, $2 \mod 4$, or $0 \mod 4$, respectively. The determinants of these matrices can be chosen to be 1 because $4u$ is relatively prime to $v$ when both $v$ is odd and $u$ is relatively prime to $v$. By \cref{rmk:ElementsOfDelta}, to be in $\rho_j(\Delta_j)$ (up to sign), it is sufficient to show that the lower diagonal entry of each matrix is congruent to $0 \mod 4$.  This is immediate for the first and the third matrices. For the second matrix, note that $u'$ has to be an odd integer since $v$ is an even integer relatively prime to $u'$. Since $v$ is congruent to $2 \mod 4$, the difference $v-2u'$ is congruent to $0 \mod 4$.
\end{proof}

We conclude this section by describing certain lifts of $N$ in $\bbH^3$ explicitly. The parametrization of these lifts play an important role in the proof of \cref{thm:UniqueThricePuncturedSphere}. We call a totally geodesic hyperplane in $\bbH^3$ \emph{vertical} if it contains $\infty$ in its circle/line at infinity. For convenience, we let $v_j := w_j^{-1} a^{(-1)^{j}}w_j^{-1} a^{-1}$ and note that $v_j y_j v_j^{-1} = a$. Observe that up to the action of $\Stab_{\Gamma_j}(\infty)$, the surface $N$ has three distinct vertical lifts in $\bbH^3$. 

A direct computation shows that 
\begin{equation*}
w_j^{-1}(0) = \infty
\quad \text{and} \quad  v_j(\tfrac{1}{2}) = \infty
\end{equation*}    
Therefore, all vertical lifts of $N$ are orbits of $\bbH^2_\bbR$, $w^{-1}_j(\bbH^2_\bbR)$, and $v_j(\bbH^2_\bbR)$ under the action of $\Stab_{\Gamma_j}(\infty) = \langle a, \ell_j\rangle$. Since the slope at $0$ (resp.~$\frac{1}{2}$) on $\bbH^2_\bbR$ is $-2/1$ (resp.~$1/0$), the vertical hyperplane $w^{-1}_j(\bbH^2_\bbR)$ (resp.~$v_j(\bbH^2_\bbR)$) has slope $-2/1$ (resp.~$1/0$) at infinity,. Furthermore, we have 
\[w_j^{-1}(\infty) = (-1)^{j+1} + \frac{2}{z_j} = -\frac{1}{2}\tau_j + 1 + (-1)^{j+1}
\] 
and 
\[v_j(\infty) = -2 \mu_j^2 + \frac{1}{2} + (-1)^{j+1}  = -\frac{1}{2}\tau_j + \frac{3}{2} +(-1)^{j+1}\]
Thus when $j$ is odd, the line at infinity of $w_j^{-1}(\bbH^2_\bbR)$ goes through the point $1$ and has slope $-2/1$. Similarly, when $j$ is odd, the line at infinity of $v_j(\bbH^2_\bbR)$ goes through the point $-\tau_j/2+5/2$ and has slope $1/0$. \cref{fig:3PVert} depicts the lines at infinity of the vertical lifts of $N$.

\begin{figure}[H]
    \centering
    \includegraphics[height=2.5in]{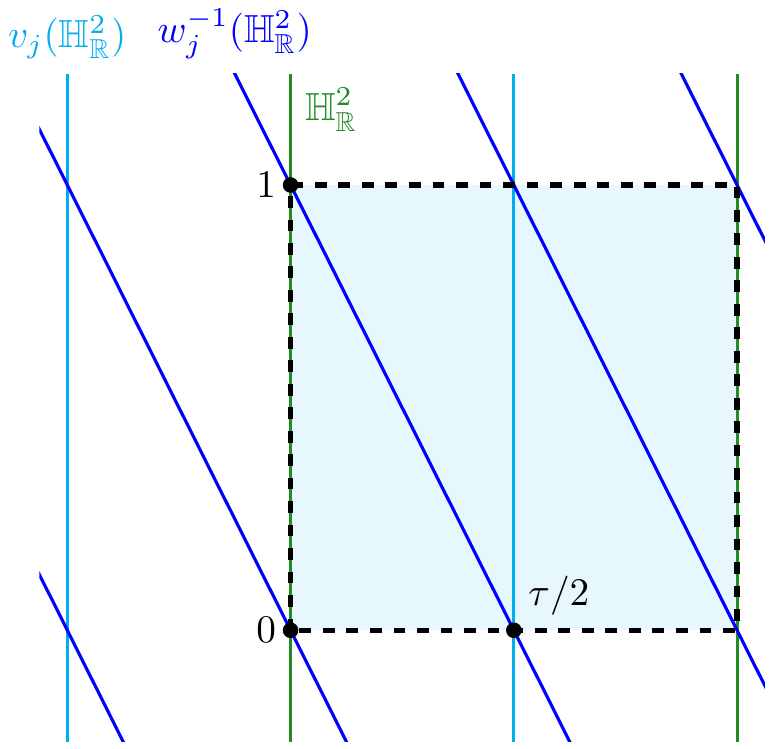}
    \caption{Vertical lifts of $N$ and a fundamental domain of the cusp when $j$ is odd}
    \label{fig:3PVert}
\end{figure}

Note that in \cref{fig:3PVert}, the element $a$ acts by translating upward and $\ell_j$ acts by translating to the right. In other words, the real axis in $\bbC$ is put in the vertical direction. Under this convention, the boundary slopes at infinity of a vertical hyperplane is precisely the slope of its line at infinity. 



\section{Main Theorems}	
\label{sec:MainThms}
For this section, we assume $j \geq 3$ and drop $j$ as the subscript notation.

\subsection{Geometric and arithmetic constraints on totally geodesic surfaces}

The following proposition of Reid \cite[Proposition 2]{Reid} gives an arithmetic constraint on the existence of totally geodesic surfaces inside hyperbolic $3$-manifolds.

\begin{prop}
\label{prop:NoClosedTG}
Let $\Gamma$ be a non-cocompact Kleinian group of finite covolume and satisfying the following two conditions:
\begin{itemize}
	\item $\bbQ(\tr \Gamma)$ is of odd degree over $\bbQ$ and contains no proper real subfield other than $\bbQ$.  
	\item $\Gamma$ has integral traces.
\end{itemize}
Then $\Gamma$ contains no cocompact Fuchsian groups and at most one commensurability class (up to conjugacy in $\PSL_2(\bbC)$) of non-cocompact Fuchsian subgroup of finite covolume.
\end{prop}	

This proposition gives a convenient way to rule out closed totally geodesic surfaces, as remarked in \cite[Lemma 5]{Reid}. 

\begin{cor}
\label{cor:NoClosedTG}
Let $M$ have an odd prime number of half-twists. Then $M$ does not contain closed immersed totally geodesic surfaces.
\end{cor}

\begin{proof}
Firstly, since $[\bbQ(\tr \Gamma):\bbQ]=j$ is an odd prime \cite[Theorem 1]{HS}, the field $\bbQ(\tr \Gamma)$ has odd degree over $\BQ$ and contains no proper subfield except $\bbQ$.
Secondly, the fact that $\text{Im}(\rho) \leq \SL_2(\bbZ[z])$ implies that $\Gamma$ has integral traces, see \cref{rem:zjPowerBasis}.
\end{proof}


To prove the main theorems, we need to rule out cusped totally geodesic surfaces that are not freely homotopic to $N$. To this end, we examine general behavior of the intersection of two totally geodesic surfaces, and then we show that any cusped totally geodesic surface in $M$ not freely homotopic to $N$ must intersect $N$ because this thrice-punctured sphere has two distinct boundary slopes.  

We have the following lemma of Fisher, Lafont, Miller and Stover \cite[Lemma 3.1]{FLMS} which describes the intersection of totally geodesic hypersurface and immersed totally geodesic submanifolds in finite-volume hyperbolic $n$-manifold. We restate their lemma for dimension $3$. A geodesic in a cusped finite-volume hyperbolic $3$-manifold $\bbH^3 /\Gamma$ is called a \emph{cusp-to-cusp geodesic} if it is the image of a geodesic in $\bbH^3$ connecting two cusp points under the action of $\Gamma$ on $\bbH^3$. 

%

\begin{lem}
\label{lem:IntersectionsOfTGS}
Let $M$ be a complete finite volume hyperbolic $3$-manifold with at least $1$ cusp. Suppose that $\Sigma_1$ and $\Sigma_2$ are two distinct properly immersed totally geodesic surfaces in $M$ such that $\Sigma_1 \cap \Sigma_2$ is non-empty. Then $\Sigma_1 \cap \Sigma_2$ is the union of closed geodesics and cusp-to-cusp geodesics.
\end{lem}

\begin{remark}
\label{rem:SelfIntersectionOfTGS}
We observe that the lemma also holds for self-intersection of a totally geodesic surface. The key point is the following. For every 2-plane in the tangent space of a point in $\bbH^3$, there exists a unique hyperplane tangent to the 2-plane. It follows that the self-intersection of totally geodesic surface is transverse. Therefore, the components of the self-intersection are unions of properly immersed complete one-dimensional submanifolds.
\end{remark}

	    
\subsection{Uniqueness of the thrice-punctured sphere}


Before we prove \cref{thm:UniqueThricePuncturedSphere} in this section, we make the following observations as well as restate that we drop the index $j$ from all notations. Let $\Sigma \subset M$ be a cusped totally geodesic surface that is not freely homotopic to $N$. 
Let $\widetilde{\Sigma}$ be a hyperplane lift of $\Sigma$ to $\bbH^3$. Since $\Sigma$ is a cusped surface, the line/circle at infinity of $\widetilde{\Sigma}$ contains a cusp point of $M$. If $\widetilde{\Sigma}$ is not a vertical hyperplane, we replace $\widetilde{\Sigma}$ by a $\Gamma$-translate that moves a cusp point of $\widetilde{\Sigma}$ to infinity. Without loss of generality, we assume that $\widetilde{\Sigma}$ is a vertical hyperplane. 

Since $\Sigma$ is distinct from $N$, we can see from \cref{fig:3PVert} that the line at infinity of $\widetilde{\Sigma}$ must intersect the line at infinity of either $\bbH^2_\bbR$ or $w^{-1}(\bbH^2_\bbR)$ transversely. Therefore, $\widetilde{\Sigma}$ intersects some vertical lift of $N$ along a geodesic $(P,\infty)$ for some $P \in \bbC$. Since $\infty$ is a cusp point of $M$, \cref{lem:IntersectionsOfTGS} implies that $(P,\infty)$ projects onto a cusp-to-cusp geodesic in $M$. In particular, $P$ is a cusp point of $M$.  

Using the geodesic $(P,\infty)$, we construct parabolic elements in $\Stab_\Gamma(\widetilde{\Sigma})$ fixing each end point and analyze the trace of the product of these elements. Since $P$ is a cusp point of $M$, there exists $\gamma \in \Gamma$ such that $\gamma(\infty) = P$. Since $P$ and $\infty$ are cusp points of $\widetilde{\Sigma}$, $\Stab_\Gamma(\widetilde{\Sigma})$ contains the nontrivial elements $a^p \ell^q$ and $\gamma a^m\ell^n \gamma^{-1}$ fixing $\infty$ and $P$, respectively, where $\gcd(p,q) = \gcd(m,n)=1$. In other words, $p/q$ (resp.~$m/n$) is the boundary slope of $\widetilde{\Sigma}$ at $\infty$ (resp.~$P$).  


The following is a critical preliminary calculation in our analysis. We note that $\tr(\rho(\pi_1\Sigma))$ must contain only real algebraic integers in $\bbQ(z)$. The condition that $\bbQ(z)$ does not contain any proper real subfield other than $\bbQ$ implies that $\tr(\rho(\pi_1 \Sigma)) \subseteq \bbZ$. Suppose that 
\begin{equation*}
\rho(\gamma) = \begin{pmatrix}
\alpha & \beta \\ \delta & \eta
\end{pmatrix}
\end{equation*}
Then, dropping the $\rho$ for convenience of notation, we must have 
\begin{equation*}
\nonumber
\tr(\gamma a^m \ell^n \gamma^{-1} a^p \ell^q) 
= (-1)^{n+q+1}[-2 + (m+n \tau)(p+q \tau) \delta^2 ] \in \bbZ 
\end{equation*}
which is true if and only if
\begin{equation}
\label{eq:TraceCondition}
\delta^2(nq\tau^2 + (mq + np)\tau + mp ) \in \bbZ
\end{equation}
We shall refer to this as the trace condition, which gives us the following:

\begin{lem}
\label{lem:SlopesOfTGS}
Let $M$ be the complement of a twist knot whose trace field is of odd degree over $\BQ$ and contains no proper real subfield other than $\BQ$.  Let $N$ be the totally geodesic thrice-punctured sphere as in \cref{fig:twistknot}. Suppose that $\Sigma$ is an immersed cusped totally geodesic surface that is not freely homotopic $N$. Then the set of boundary slopes of $\Sigma$ is precisely $\{1/0,-2\}$, up to multiplicity.
\end{lem}

\begin{proof}
It follows from the previous discussion that $\Sigma$ has a vertical lift $\widetilde{\Sigma}$ to $\bbH^3$. The lift $\widetilde{\Sigma}$ must intersect either $\bbH^2_\bbR$ or $w^{-1}(\bbH^2_\bbR)$, whose respective slopes are $1/0$ and $-2/1$, along some vertical geodesic $\theta$. One end point of $\theta$ is $\infty$, so by \cref{lem:IntersectionsOfTGS}, $\theta$ covers a cusp-to-cusp geodesic. This implies that the other end point of $\theta$ must also be a cusp point in $\bbH^2_\bbR$ or $w^{-1}_j(\bbH^2_\bbR)$. Therefore, either $\theta = (s/t, \infty)$ for $s/t \in \bbQ$ or $\theta = (w^{-1}(s/t), \infty)$ for $s/t \in \bbQ \cup\{\infty\}$. 

Suppose that the boundary slopes on $\widetilde{\Sigma}$ are $p/q$ at $\infty$ and $m/n$ at the other end point of $\theta$. As boundary slopes, neither $(m,n)$ nor $(p,q)$ is $(0,0)$. We have nontrivial elements $a^p\ell^q$ and $\gamma a^m\ell^n\gamma^{-1}$ in $\Stab_{\Gamma}(\widetilde{\Sigma})$ where $\gamma \in \Gamma$ has the property that $\gamma^{-1}(s/t) = \infty$ or $\gamma^{-1}w^{-1}(s/t) = \infty$. We consider these two cases separately, and use the general form
    \begin{equation*}
    \gamma = \begin{pmatrix}
    \alpha & \beta \\ \delta & \eta 
    \end{pmatrix}
    \end{equation*}
Note that the end points of $\theta$ are distinct since $\theta$ is a geodesic in $\bbH^3$. Consequently, $\gamma$ cannot fix $\infty$, and so $\delta \neq 0$.

\begin{description}
\item[Case 1:] Suppose that $\widetilde{\Sigma}$ intersects $\bbH_\bbR^2$ nontrivially along a geodesic $\theta = (s/t,\infty)$ for some $s/t \in \bbQ$, as seen in \cref{fig:3PVert_Sigma}. In view of \cref{lem:CuspsOrbits}, we may choose $\gamma$ to be either
\begin{equation*}
	\begin{pmatrix}
	t' & s \\ -4s' & t 
	\end{pmatrix} w, \quad 
	\begin{pmatrix}
	s - 2t' & t' \\ t - 2s' & s' 
	\end{pmatrix} 	v^{-1}, \text{ or } \begin{pmatrix}
	s & t' \\ t & s'
	\end{pmatrix}
\end{equation*}
according to whether $s/t$ belongs to $[0]$, $[\frac{1}{2}]$ or $[\infty]$, respectively. It suffices to compute $\delta$ in these three circumstances by using the trace condition \cref{eq:TraceCondition}. 

Since $v^{-1}(\infty) = \frac{1}{2}$, we see that $v^{-1}$ is of the form
\begin{equation*}
\pm \begin{pmatrix} 1 & * \\ 2 & * \end{pmatrix}    
\end{equation*}
This implies that $\delta = t \in \bbZ \smallsetminus \{0\}$ in the last two situations. We show that this contradicts the assumption that $(p,q) \neq (0,0) \neq (m,n)$. It follows from \eqref{eq:tau} that the trace field is the same as the cusp field. That is, $\bbQ(z) = \bbQ(\tau)$. Then 
\[ [\bbQ(\tau) :\bbQ] = [\bbQ(z) :\bbQ] \geq 3,\] 
so the minimal polynomial of $\tau$ over $\BQ$ has degree $\geq 3$. In particular, $\tau$ cannot satisfy the quadratic polynomial over $\bbZ$ given by the trace condition. Therefore, we have $nq = 0$ and $mq + np =0$ since $\delta^2 \neq 0$. Since $\widetilde{\Sigma}$ intersects $\bbH^2_\bbR$ nontrivially, the boundary slope at $\infty$ on $\widetilde{\Sigma}$ is not meridional; in other words, $q \neq 0$. We thus have $n=0$, which implies that $mq = 0$, and so $m=0$. This contradicts the assumption that $(m,n)\neq (0,0)$.  

In the first situation, we get $\delta = t \mu z$ and so $\delta^2 = -t^2 z$. The trace condition becomes
\begin{align*}
	\delta^2(nq\tau^2 + (mq + np)\tau + mp ) 
	& \in \bbZ
	\\
	z [nq(-4z^{-1}+2)^2 +(mq + np)(-4z^{-1}+2) + mp ] 
	& \in \bbQ
	\\
	z [16nqz^{-2}  -(16nq +4mq +4np )z^{-1}+  4nq +2mq +2 np + mp ]
	& \in \bbQ
	\\
	16nqz^{-1} + ( 2n +m)(2q+p)z 
	& \in \bbQ 
	\\
	16nq + ( 2n +m)(2q+p)z^2 
	& \in \bbQ z
\end{align*}
Since the minimal polynomial of $z$ has degree at least $3$, satisfying the above is equivalent to satisfying
\begin{equation*}
16 nq 
= (2n+m)(2q+p)
= 0
\end{equation*}
Again, since $\widetilde{\Sigma}$ intersects $\bbH^2_\bbR$ nontrivially, we have $q \neq 0$, so $m/n=1/0$ and $p/q = -2/1$.

\begin{figure}[H]
    \centering
    \includegraphics[height=2.5in]{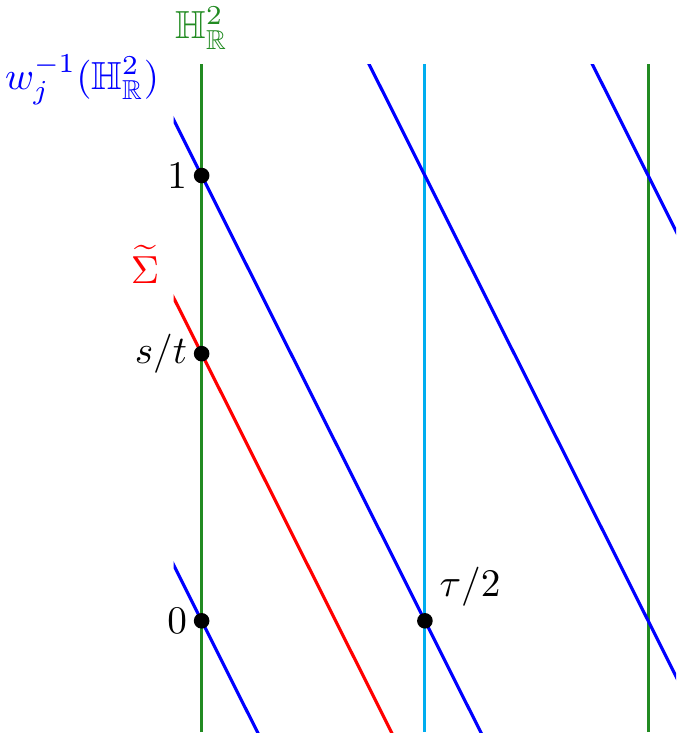}
    \caption{Vertical lifts of $N$ and $\Sigma$ when $j$ odd}
    \label{fig:3PVert_Sigma}
\end{figure} 

\item[Case 2:] Suppose that $\widetilde{\Sigma}$ is parallel to $\bbH^2_\bbR$, or equivalently $\theta = (w^{-1}(s/t),\infty)$ for some $s/t \in \bbQ$. We may assume that $p/q = 1/0$. In view of \cref{lem:CuspsOrbits}, we may choose $\gamma = \begin{pmatrix}
\alpha & \beta \\ \delta & \eta 
\end{pmatrix}$ to be either 
\begin{equation*}
	w^{-1}\begin{pmatrix}
	-t' & s \\ -4s' & t 
	\end{pmatrix} w, \text{ or }
	w^{-1}\begin{pmatrix}
	s - 2t' & t' \\ t - 2s' & s' 
	\end{pmatrix} 	v^{-1}, \text{ or } w^{-1}\begin{pmatrix}
	s & t' \\ t & s'
	\end{pmatrix}
\end{equation*}
When $s/t \in [0]$, we get $\delta = sz$, so
\begin{align*}
	\delta^2 (m+n \tau) 
	& \in \bbZ
	\\
	s^2z^2 (m+n(-4z^{-1}+2)) 
	& \in \bbZ
\end{align*}
This is satisfied if and only if $s^2n = s^2(2n+m) = 0$ because the minimal polynomial of $z$ is at least cubic over $\bbQ$. If $s = 0$, then $w^{-1}(0) = \infty$ since endpoints of $\theta$ must be distinct, therefore $s \neq 0$.  When $s \neq 0$, we must have $m=n=0$, which is a contradiction.

When $s/t \in [\frac{1}{2}]$ or $[\infty]$, we see that $\delta = -s\mu z$ and use an equivalent calculation to Case 1 to conclude that the boundary slopes are $p/q = 1/0$ and $m/n = -2/1$.
\end{description}
\vspace{-2em}
\end{proof}



We now prove the uniqueness of $N$. 

\begin{proof}[Proof of \cref{thm:UniqueThricePuncturedSphere}]
Throughout the proof, we assume that the trace field of $\Gamma$ has odd degree and has no proper real subfield except for $\bbQ$. In particular, $\Gamma = \Gamma_j$ for some positive odd integer $j$. Since $\Gamma$ has integral traces by \cref{cor:NoClosedTG} and its trace field satisfies the condition in \cref{prop:NoClosedTG}, any totally geodesic surface in $M$ must have at least one cusp.

\cref{lem:SlopesOfTGS} has the following consequence. If $\Sigma$ is a cusped totally geodesic surface in $M$ that is not $N$, then there exists a vertical lift $\widetilde{\Sigma}$ in $\bbH^3$ with slope $-2/1$ at $\infty$. This lift intersects $\bbH^2_\bbR$ along a cusp-to-cusp geodesic $(s/t, \infty)$. A careful look at the first case in the proof of \cref{lem:SlopesOfTGS} shows that $s/t$ must belong to the orbit of $[0]$ under the action of $\Delta$ on $\bbH^2_\bbR$; in particular, $t$ is an odd integer. Using the action of $a$, we may assume that $0 < s/t < 1$, since $\Sigma$ is distinct from $N$. Let $H_{s/t}$ be the vertical lift of $\Sigma$ to $\bbH^3$ of slope $-2/1$ at $\infty$ that intersects the hyperplane $\bbH^2_\bbR$ along the geodesic $(s/t,\infty)$. 

In order to prove \cref{thm:UniqueThricePuncturedSphere}, it suffices to show that each vertical hyperplane $H_{s/t}$ for $0<s/t<1$ with $t$ odd covers an infinite area totally geodesic surface, which implies that $H_{s/t}/\Stab_{\Gamma}(H_{s/t})$ must be dense in $M$. With this density in mind, we construct a non-closed non-cusp-to-cusp geodesic in $M$ as a self-intersection of $H_{s/t}/\Stab_{\Gamma}(H_{s/t})$.

Fix a rational number $0<s/t<1$ where $t$ odd and $s$, $t$ relatively prime. Suppose that $H_{s/t}$ covers a complete properly immersed finite area totally geodesic surface. We move $H_{s/t}$ by an element of $\Gamma$ that takes $s/t$ to $\infty$. To write down such an element, we first map $s/t$ to $0$ by an element of $\Delta$, then map $0$ to $\infty$ by $w^{-1}$. Consider the following element 
\begin{equation*}
	\zeta = w^{-1}\begin{pmatrix}
	t & -s \\ 4s' & -t'
	\end{pmatrix}
\end{equation*}
taking $s/t$ to $\infty$. Note that the latter matrix in the product belongs to $\Delta$ as discussed in \cref{rmk:ElementsOfDelta}. Therefore, $\zeta \in \Gamma$. Since $\zeta(s/t) =\infty$, the hyperplane $\zeta(H_{s/t})$ must be vertical. 

The analysis in the first case of \cref{lem:SlopesOfTGS} shows that the slope of $H_{s/t}$ at $s/t$ is $1/0$.  Consequently, the slope of $\zeta(H_{s/t})$ is at $\infty$ is $1/0$. Since $H_{s/t}$ has slope $-2/1$, the vertical hyperplane $\zeta(H_{s/t})$ intersects $H_{s/t}$ along a geodesic $\theta = (\theta_-,\infty)$ for some $\theta_-\in \bbC$. We have 
\[\zeta(\infty) =1 +  \left(\frac{1}{2}+\frac{s'}{t}\right)\frac{4}{z}\]
Since the slope of $H_{s/t}$ at $\infty$ is $-2/1$, the line at infinity of $H_{s/t}$ is parametrized by 
\[\frac{s}{t} + r(\tau - 2) = \frac{s}{t} - r\frac{4}{z}\] for some $r\in\bbR$ while that of $\zeta(H_{s/t})$ is parametrized by 
\[\zeta(\infty) + r'\] 
for some $r'\in\bbR$. It follows that 
\begin{equation}
\label{eq:theta-}
\theta_- = \frac{s}{t} +   \left(\frac{1}{2}+\frac{s'}{t}\right)\frac{4}{z} = \frac{ s z + 4s' + 2t }{t z}. 
\end{equation}
Since one end point of $\theta$ is at $\infty$, a cusp point of $\Gamma$, \cref{rem:SelfIntersectionOfTGS} implies that $\theta$ covers a cusp-to-cusp geodesic. In particular, $\theta_-$ is a cusp point of $\Gamma$. Thus, there exists 
\[\gamma = \begin{pmatrix}
\alpha & \beta \\ \delta & \eta 
\end{pmatrix} \in \Gamma \leq \SL_2(\bbZ[z])\] 
such that $\gamma(\theta_-) = \infty$, see \cref{rem:zjPowerBasis}. We must have 
\begin{equation}
\label{eq:eta}
	\eta = -\delta \theta_-
\end{equation}
By \cref{lem:SlopesOfTGS}, the slope at $\theta_-$ of $H_{s/t}$ is either $1/0$ or $-2/1$. We consider these two cases separately.
\begin{description}
\item[Case 1:] 
Suppose that the slope at $\theta_-$ is $1/0$. Consequently, $\Stab_\Gamma(H_{s/t})$ contains $\gamma^{-1}a\gamma$. The trace condition \cref{eq:TraceCondition} implies that
\begin{equation*}
	\tr(\gamma^{-1} a\gamma
	a^2\ell^{-1}) = \delta^2(2-\tau) = -4\mu^2\delta^2 \in \bbZ.
\end{equation*}
Let us write $\delta^2 = z \, k/4$ for some $k \in \bbZ$. Note that $\delta \neq 0$ as $\Gamma$ contains no trace $0$ element. Since $\delta\in\bbZ[z]$, we have $k/4 \in \bbZ$. Using \cref{eq:zj_muj}, we have $\delta = \pm \mu^{-1} \sqrt{-k/4}$ or $\mu\delta = \pm \sqrt{-k/4}$. Since $[\bbQ(z):\bbQ]$ is odd by \cite[Theorem 1]{HS}, $\bbQ(z)$ does not contain the square root of non-square integer. It follows that $\sqrt{-k/4} \in \bbZ$. In other words, $\delta = \pm  m \mu^{-1}$ for some $m \in \bbZ$. Substituting this and \cref{eq:theta-} into \cref{eq:eta}, we have 
\begin{equation*}
\pm\eta\mu^{-1} = -m(-z)\frac{sz + 4s' + 2t}{tz}  = \frac{ms}{t}z + \frac{m(4s' + 2t)}{t} \in \bbZ[z].
\end{equation*}
The integer $t$ must divide $m$ because $s$ and $t$ are relatively prime.
Since $\mu$ is a unit in $\bbZ[z]$, we have the following containment of ideals $\eta \bbZ[z] \subseteq \frac{m}{t} \bbZ[z]$ and $\delta\bbZ[z] \subseteq m\bbZ[z] \subseteq \frac{m}{t}\bbZ[z]$. In particular, the ideal $(\delta,\eta)\bbZ[z]$ generated by $\delta$ and $\eta$ in $\bbZ[z]$ is contained in $\frac{m}{t}\bbZ[z]$. Since $\alpha\eta - \beta\delta = 1$, $(\delta,\eta)\bbZ[z] = \bbZ[z]$. This fact implies that $m = \pm t$. We obtain the final expressions for $\eta$ and $\delta$.  
\begin{equation}
\label{eq:eta_delta_case1}
\eta =\pm(sz + 4s' + 2t) \mu \quad \text{and} \quad \delta = \pm t\mu^{-1}
\end{equation}
The slope of $H_{s/t}$ at $\theta_-$ is $1/0$ and $\gamma(\theta_-) = \infty$, so $\gamma(H_{s/t})$ is a vertical hyperplane with $1/0$ slope at $\infty$. Thus, this hyperplane intersects $w^{-1}(\bbH^2_\bbR)$ along a vertical geodesic $\sigma$. By \cref{lem:IntersectionsOfTGS}, $\sigma$ covers a cusp-to-cusp geodesic in $M$, since one end point of $\sigma$ is at $\infty$. Therefore, the other end point of $\sigma$ must be a cusp point of $\Gamma$. 

The boundary of $w^{-1}(\bbH^2_\bbR)$ is parametrized by $1 + r'(2-\tau)$ for $r'\in\bbR$, and the boundary of $\gamma(H_{s/t})$ is parametrized by $\gamma(\infty) + r$ for some $r\in \bbR$. Since cusp points of $\Gamma$ are contained in $\bbQ(z)\cup \{\infty\}$, there exist $r,r' \in \bbR$ such that
\begin{equation*}
\pm \frac{\alpha}{t}\mu + r = 1 +  r'(2-\tau) = 1+ \frac{4r'}{z}\in \bbQ(z)
\end{equation*} 
Therefore, $r,r'\in \bbQ(z) \cap \bbR =\bbQ$ because $\bbQ(z)$ has no proper real subfield except $\bbQ$. Using \cref{eq:zj_muj}, we see that
\begin{equation*}
\pm\alpha/\mu =(4r't +(1-r) tz)
\end{equation*}
Since $\alpha\in\bbZ[z]$ and $\mu$ is a unit of $\bbZ[z]$, the right-hand side of the equation must belong to $\bbZ[z]$. In particular, $4r't$ and $(1-r)t$ are integers. To simplify our notation, we write $\alpha = (\alpha_1 z + \alpha_0)\mu$ where $\alpha_i \in \bbZ$. Now we substitute the expression for $\alpha$, $\eta$, and $\delta$ (see \cref{eq:eta_delta_case1}) into the determinant of $\gamma$ to obtain 
\begin{equation*}
\begin{aligned}
(\alpha_1z + \alpha_0)(sz + 4s'+2t)\mu^2 \pm t \mu^{-1} \beta &= 1\\
(\alpha_1 z + \alpha_0)(sz + 4s'+2t) \pm t \mu \beta &= -z
\end{aligned}    
\end{equation*}
Since $\mu,\beta\in\bbZ[z]$, we can write the equation above as
\begin{equation*}
\alpha_0(4s' + 2t)  + [\alpha_1(4s'+2t) + \alpha_0 s ]z + \alpha_1s z^2 + t\sum_{i=0}^{j-1} k_i z^i = -z
\end{equation*}
for $k_i \in \bbZ$. Noting that $[\bbQ(z):\bbQ] \geq 3$, the equality above implies that 
\begin{equation*}
\begin{aligned}
\alpha_0(4s' + 2t) + tk_0 &= 0, \\
\alpha_1(4s'+ 2t) + \alpha_0s + tk_1 &= -1, \\
(\alpha_1 s + tk_2) &= 0.
\end{aligned}
\end{equation*}
Both $s$ and $s'$ are units in $\bbZ/t\bbZ$, so the first and the third equations imply that 
\[\alpha_i \equiv 0 \mod t.\] 
This contradicts the second equation. Therefore, there exists no such $\gamma \in \Gamma$ in this case. 

\item[Case 2:]
Suppose that the slope at $\theta_-$ is $-2/1$. It follows that $\Stab_\Gamma(H_{s/t})$ contains $\gamma^{-1} a^2\ell_j^{-1}\gamma$. The trace condition \cref{eq:TraceCondition}, implies that
\begin{equation*}
	\tr(\gamma^{-1} a^2\ell^{-1}\gamma
	a^2\ell^{-1}) =\delta^2(2-\tau)^2 = \frac{16\delta^2}{z^2}\in \bbZ. 
\end{equation*}
Let us write $\delta^2 = z \, k/16$ for some $k\in\bbZ$. Since $\delta \in \bbZ[z]$, we must have $k/16 \in \bbZ$. Now we write $\delta = \pm z \sqrt{k/16}$. Note that $\delta \neq 0$ since $\Gamma$ does not contain element with trace $0$. Since $[\bbQ(z):\bbQ]$ is odd by \cite[Theorem 1]{HS}, the trace field $\bbQ(z)$ does not contain the square root of any non-square integer. It follows that $\sqrt{k/16} \in \bbZ$. In other words, $\delta = m z$ for some $m \in \bbZ$. Substituting this into \cref{eq:eta}, we have 
\begin{equation*}
	\eta = -mz\frac{s z + 4s' + 2t}{tz}
	=-\frac{ms}{t}z - \frac{m(4s'+2t)}{t} \in \bbZ[z].
\end{equation*}
Since $\eta \in \bbZ$, we see that $ms/t \in \bbZ$. Since $s$ and $t$ are relatively prime, $t$ divides $m$. We have the following containment of ideals $\eta \bbZ[z] \subseteq \frac{m}{t}\bbZ[z]$ and $\delta \bbZ[z] =  m \bbZ[z] \subseteq \frac{m}{t}\bbZ[z]$. In particular, $(\delta,\eta)\bbZ[z] \subseteq \frac{m}{t}\bbZ[z]$. Since $\alpha \eta - \beta \delta = 1$, $(\delta,\eta)\bbZ[z] = \bbZ[z]$. This implies that $m = \pm t$. We obtain the following expressions for $\eta$ and $\delta$

\begin{equation}
\label{eq:eta_delta_case2}
\eta = \pm(sz +4s' + 2t)\quad \text{and}\quad \delta = \pm t z.
\end{equation}

Since the slope of $H_{s/t}$ at $\theta_-$ is $-2/1$ and $\gamma(\theta_-) = \infty$, we observe that $\gamma(H_{s/t})$ is a vertical plane with slope $-2/1$ at $\infty$. It must intersect the hyperplane $\bbH^2_\bbR$ at a vertical geodesic $\sigma$. By \cref{lem:IntersectionsOfTGS}, $\sigma$ covers a cusp-to-cusp geodesic in $M$, since one end point of $\sigma$ is $\infty$. The other end point of $\sigma$ is a cusp point on the line at infinity of $\bbH^2_\bbR$ and so must be a rational number. In other words, there exists a real number $r$ such that 
\begin{equation*}
	\frac{\alpha}{tz} + r\frac{4}{z} \in \bbQ.
\end{equation*}
This implies that $r \in \bbQ(z) \cap \bbR = \bbQ$ since $\bbQ(z)$ has no proper real subfield except $\bbQ$. Therefore $\alpha = \alpha_1 z + \alpha_0 \in \bbZ[z]$. Now we substitute the expression for $\alpha$, $\eta$ and $\delta$ (see \cref{eq:eta_delta_case2}) into the determinant of $\gamma$ to obtain 
\begin{equation*}
	(\alpha_1 z + \alpha_0)(sz + 4s' + 2t) - t z \beta  = \pm 1 \mod t
\end{equation*}  
Since $\mu,\beta\in\bbZ[z]$, we can write the equation above as
\begin{equation*}
\alpha_0(4s' + 2t) + [\alpha_1(4s' + 2t)+\alpha_0 s] z + \alpha_1 s z^2 + t\sum_{i=0}^{j-1} k_i z^i = \pm 1     
\end{equation*}
Since $[\bbQ(z):\bbQ] \geq 3$, the equality above implies that 
\begin{equation*}
\begin{aligned}
\alpha_0(4s' + 2t) + tk_0 &= \pm 1 \\
\alpha_1(4s' + 2t) + \alpha_0 s + tk_1 &= 0 \\
\alpha_1 s + t k_2 &= 0  
\end{aligned}
\end{equation*}
Since $s$ is a unit in $\bbZ/t\bbZ$, the third equation implies that $\alpha_1 \equiv 0 \mod t$. Given that $\alpha_1 \equiv 0 \mod t$ and $s \in (\bbZ/t\bbZ)^\times$, the second equation implies that $\alpha_0 \equiv 0 \mod t$. These conditions on $\alpha_i$ contradict the first equation. Consequently, there exists no such $\gamma$ in this case. 
\end{description}
The analysis of both cases shows that $\theta_-$ is not a cusp point of $\Gamma$ because there is no $\gamma\in \Gamma$ such that $\gamma(\infty) = \theta_-$. This contradiction implies that $H_{s/t}$ does not cover a finite area totally geodesic surface. This completes the proof of \cref{thm:UniqueThricePuncturedSphere}.
\end{proof}

The existence of infinitely many twist knot complements with a unique totally geodesic surface follows as a consequence.

\begin{proof}[Proof of \cref{cor:ExistUnique}]
Let $M$ have an odd prime number of twists.  Note that $\tr(\rho(\pi_1\Sigma)) \subset \bbZ$ because $\Gamma$ has integral traces, $\bbQ(\tr \Gamma)$ has no proper real subfield (except $\bbQ$) because it has odd prime degree over $\BQ$ by \cite[Theorem 1]{HS}, and $\rho(\pi_1\Sigma)$ is a Fuchsian subgroup.  Thus the conditions of \cref{thm:UniqueThricePuncturedSphere} are satisfied.

Using Magma \cite{Magma} and SageMath \cite{sage}, we verify that $\Gamma_j$ also satisfies the conditions of \cref{thm:UniqueThricePuncturedSphere} for odd integers with $3\leq j \leq 99$, see \cite{Codes} for the code.

Twist knot complements with distinct numbers of half-twists are always non-commensurable because their invariant trace fields $\bbQ(z_j)$ are pairwise distinct. Thus we have a family of infinitely many non-commensurable $3$-manifolds with exactly one totally geodesic surface.
\end{proof}

\subsection{Lifting the thrice-punctured sphere}

We at last deduce the proof of \cref{thm:kTGSurfaces}, reinstating the index j.

\begin{proof}[Proof of \cref{thm:kTGSurfaces}]
In view of \cref{cor:ExistUnique}, to obtain a finite volume hyperbolic $3$-manifold with exactly $k$ totally geodesic surfaces, it suffices to find covers of $M_j$ such that $N$ has exactly $k$ lifts for some odd prime $j$. We construct these manifolds from dihedral covers of the twist knot complements. We first count the number of lifts of $N$ to dihedral covers of twist knot with an odd number of half-twists. For this purpose, let $j$ be a positive odd integer. 

We write down an explicit surjective homomorphism from $\Gamma_j$ to the dihedral group $D_n$ of order $2n$. The group $D_{n}$ has the following presentation 
\begin{equation*}
    D_n = \langle \alpha,\beta \mid \alpha^2, \beta^2, (\alpha\beta)^n \rangle.
\end{equation*}
Let us fix an integer $n > 1$ dividing $2j+1$. We have a surjective homomorphism $\varphi:\Gamma_j \to D_n$ defined by $\varphi(a) =\alpha$ and $\varphi(b) =\beta$ since $\varphi(w_jaw_j^{-1}b^{-1}) = \varphi(ab)^{2j+1}$. Let $\hat{M}_j$ be the cover corresponding to $\ker \varphi$. 

We observe that 
\begin{equation*}
    \varphi(y_j) = \varphi((aw_j)^2a(aw_j)^{-2}) =\alpha   
\end{equation*}
since $\varphi(aw_j)$ has order 2. Therefore, $\varphi(\Delta_j) = \langle \alpha \rangle$ is a cyclic subgroup of order 2. Let $N_0$ be a lift of $N$ to $\hat{M}_j$ corresponding to $\ker \varphi|_{\Delta_j}$. Since $\hat{M}_j$ is a regular cover, the deck group $D_n$ acts transitively on the lifts of $N$ to $\hat{M}_j$. We identify these lifts with the orbit of $N_0$ under the action of $D_n$. Since the stabilizer of $N_0$ under the action of $D_n$ is $\langle \alpha \rangle$, we identify the lifts of $N$ to $\hat{M}_j$ with the left cosets of $\langle \alpha \rangle$ in $D_n$. In particular, $N$ lifts to $n$ distinct totally geodesic surfaces $\left\{ N_{m} \right\}_{m=0}^{n-1}$ in $\hat{M}_j$. 

Observe that the cyclic subgroup $\langle \alpha \beta \rangle$ of order $n$ forms a complete left coset representatives of $\langle\alpha\rangle$ in $D_n$. Up to reordering, we identify $N_m$ with the coset $(\alpha\beta)^m\langle\alpha\rangle$ for $0 \leq m \leq n-1$. The action of $\alpha$ on the lifts $N$ is given by 
\[
\alpha \cdot N_m = N_{n - m},
\]
where $N_n := N_0$, since $\alpha \cdot (\alpha\beta)^m \langle \alpha \rangle =   (\alpha\beta)^{n - m} \langle \alpha \rangle$. Since $n$ is an odd integer, $\alpha$ leaves invariant exactly one surface $N_0$ and permutes $n-1$ surfaces pairwise. It follows that $\hat{M}_j/\langle \alpha \rangle$ contains exactly $1 + \frac{n-1}{2} = (n+1)/2$ lifts of $N$. 

Any integer $k\geq 2$ can be written as $(n+1)/2$ where $n=2k-1 \geq 3$. To prove the theorem, it suffices to show that for any odd integer $n$ there exist infinitely many primes $p$ such that $n$ divides $2p+1$. Let $n=2q+1$ and thus $\gcd(q,n)=1$. By Dirichlet's theorem, the arithmetic progression $\{q+nk\}_{k=0}^\infty$ contains infinitely many primes $p_i = q + nk_i$. Note that $n$ divides $2p_i+1$ for all $i$. This completes the proof of \cref{thm:kTGSurfaces}. 
\end{proof}

\subsection{Non-elementary Fuchsian subgroups}

The proof of \cref{lem:SlopesOfTGS} gives infinitely many explicit examples of non-elementary maximal Fuchsian subgroups of infinite covolume in $\Gamma_j$. 

\begin{prop}
\label{prop:ExistenceOfBadStab}
Let $M$ be a twist knot complement with at least $3$ half-twists. Let $H_{s/t}$ be the vertical hyperplane intersecting $\bbH^2_\bbR$ at $s/t \in \bbQ$ with slope $-2/1$ at $\infty$. For all odd $t$, the stabilizer of $H_{s/t}$ in $\Gamma$ is non-elementary. 
\end{prop}

\begin{proof}
	We observe that $\Stab_{\Gamma}(H_{s/t})$ contains a subgroup generated by 
	\begin{equation*}
		a^{-2}\ell \quad \text{and} \quad \begin{pmatrix}
		-t' & s \\ -4s' & t
		\end{pmatrix} b \begin{pmatrix}
		t & -s \\ 4s' & -t'
		\end{pmatrix}
	\end{equation*}
where $4ss' - tt' = 1$. The fact that the latter element belongs to $\Gamma_j$ follows from \cref{rmk:ElementsOfDelta}. It is clear that $a^{-2}\ell$ stabilizes $H_{s/t}$, so the line at infinity of $H_{s/t}$ has slope $-2/1$ and thus is parametrized by 
\[\frac{s}{t} + r(2-\tau) \]
for $r \in \bbR$. We have
	\begin{equation*}
		\begin{pmatrix} -t' & s \\ -4s' & t \end{pmatrix} b \begin{pmatrix} t & -s \\ 4s' & -t' \end{pmatrix}\left(\frac{s}{t} + \frac{4r}{z}\right) = \frac{s}{t} + \frac{r}{4rt^2+1}\frac{4}{z} = \frac{s}{t} + \frac{r}{4rt^2 + 1}(2-\tau)
	\end{equation*}
for any $r\in\bbR$. Since $r/(4rt^2 + 1) \in \bbR$, the image point also belongs to the line at infinity of $H_{s/t}$. Therefore, the latter element stabilizes $H_{s/t}$. 

These are independent parabolic element with respect to the action of the stabilizer of $H_{s/t}$ on $H_{s/t}$ because they have distinct fixed points. They thus generate a non-elementary subgroup of the stabilizer of $H_{s/t}$.
\end{proof}

Since all vertical lifts of $N$ are described at the end of \cref{sec:Prelim}, for $s/t \in \bbQ  \smallsetminus \bbZ$ and odd $t$, the image of $H_{s/t}$ in $M_j$ is not freely homotopic to $N$. In particular, $\Stab_{\Gamma_j}(H_{s/t})$ and $\Delta_j$ are not conjugate in $\Gamma_j$. By \cite[Theorem 1.1]{BFMS}, at most finitely many $\Gamma_j$-orbits of hyperplanes $H_{s/t}$ cover a finite area totally geodesic surface. Other quotients of these hyperplanes must be of infinite area. Furthermore, if the trace field of $M_j$ satisfies the condition in \cref{thm:UniqueThricePuncturedSphere}, then \cref{thm:UniqueThricePuncturedSphere} says that $\Delta_j$ is the unique (up to conjugacy) finite covolume maximal Fuchsian subgroup of $\Gamma_j$; that is,

\begin{cor}
\label{cor:ExistenceOfBadStab}
Let $j$ be an odd integer such that $\bbQ(z_j)$ does not contain any proper real subfield. Then for all $s/t \in \bbQ \smallsetminus \bbZ$ and odd $t$, $\Stab_{\Gamma_j}(H_{s/t})$ is a non-elementary Fuchsian group of infinite covolume.
\end{cor}

In other words, \cref{cor:ExistenceOfBadStab} provides examples of infinitely many hyperbolic $3$-manifolds that contain non-elementary  Fuchsian subgroups of infinite covolume. This behavior demonstrates a drastic difference between arithmetic and non-arithmetic hyperbolic manifolds. In an arithmetic hyperbolic $3$-manifold, every non-elementary  Fuchsian subgroup has finite covolume \cite[Theorem 4]{MaclachlanReid87}. Infinite covolume Fuchsian groups of this kind are known to exist, such as using results in \cite{FLMS}, but we do not know of any explicit examples in the literature.

\subsection{Right-angled knot complements}
\label{Right-angledKnotComplement}
We now discuss an application of our boundary slope analysis. A hyperbolic manifold is called \emph{right-angled} if it can be constructed by gluing together a set of hyperbolic polyhedra whose dihedral angles are all $\pi/2$.  Being right-angled is ever-present in $3$-dimensional topology, such as showing that a manifold is LERF \cite[Corollary 1.4]{CDW} and addressing the virtually fibered conjecture \cite[Corollary 1.2]{CDW}.

It is known the the faces of these polyhedra give rise to immersed totally geodesic surfaces \cite{CKP}.  In the case of a cusped hyperbolic $3$-manifold, at least one polyhedron must admit an ideal vertex.
The intersection of this polyhedron with a sufficiently small horosphere based at the ideal vertex must be a Euclidean right-angled polygon.  Sides of this polygon being perpendicular may be checked by looking at the angle between two intersecting hyperplanes.



There are examples in \cite{C} of knot complements that are not right-angled because they contain no totally geodesic surface. It is also known that all knots up to $11$ crossings are not right-angled because of volume obstructions \cite{CKP}.  
However, we authors do not know of any existing literature that provides an infinite list of knot complements that are not right-angled.  We give the first such infinite family.

\begin{proof}[Proof of \cref{cor:NoRightAngled}]
For $j$ odd prime, we check that the right-angled property fails by finding the angle between the real hyperplane through $0$ and the hyperplane of slope $-2/1$ through $0$. This angle is $\pi/2$ if and only if the argument of $-1+\frac{\tau}{2}$ is $\pi/2$.

\begin{figure}[H]
    \centering
    \includegraphics[height=2.5in]{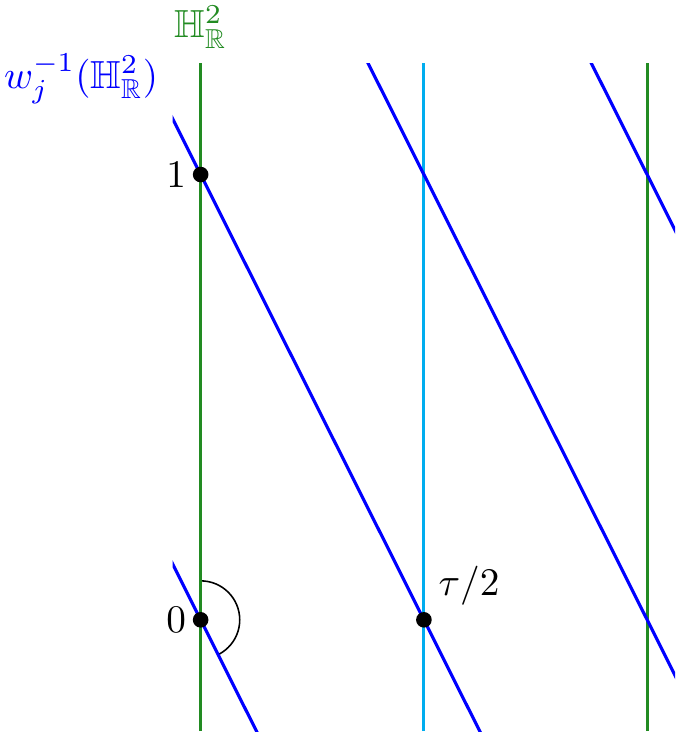}
    \caption{Vertical lifts of $N$ for $j$ odd along with the angle that must be $\pi/2$}
    \label{fig:3PVert_right}
\end{figure}

We know that $\tau=-4z^{-1}+2$, so $-1+\frac{\tau}{2} = -2z^{-1}$, whose argument is bounded away from $\pi/2$ by \cite[Theorem 1]{HS}.
Therefore, no twist knot complement of this form is right-angled.
\end{proof}

\section{Totally geodesic surfaces in the figure-8 knot complement}
\label{sec:Fig8Knot}

The figure-8 knot complement $M$ is exceptional among twist knot complements.  It is the sole arithmetic knot complement \cite{ReidThesis} and therefore admits infinitely many totally geodesic surfaces \cite{MaclachlanReid}.  Two such surfaces are shown in \cref{fig:3PPair}.  One is the previously defined $N$.  The other is a dual surface $N'$ whose boundary slopes are $1/0$ and $2/1$; this information can be calculated using the techniques in \cref{sec:Prelim} and visualized in \cref{fig:3PPair_vert}.

\begin{figure}[h]
    \centering
    \subfloat{\includegraphics[width=3in]{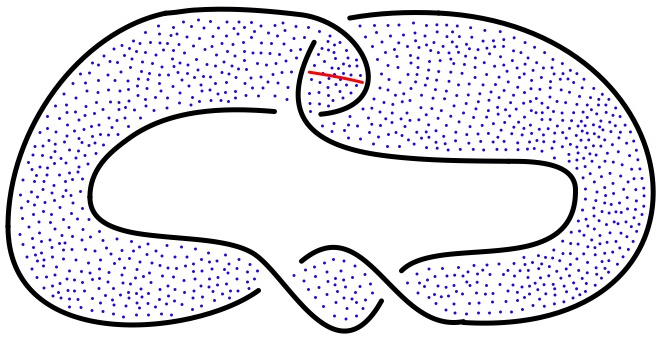}}
    \hspace{1cm}
    \subfloat{\includegraphics[width=3in]{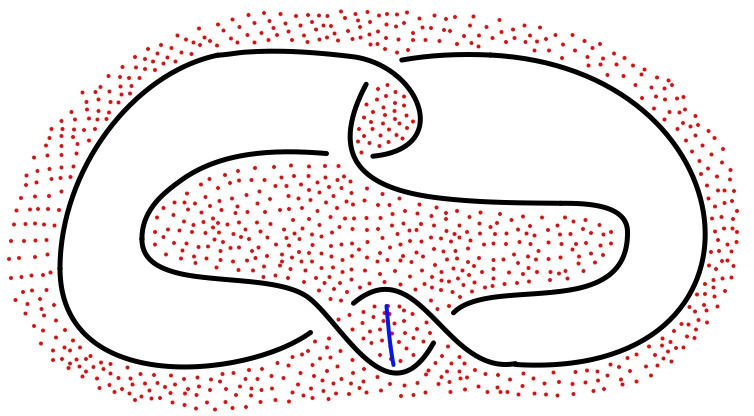}}
    \caption{Two distinct totally geodesic thrice-punctured spheres $N$ and $N'$ in the figure-8 knot complement}
    \label{fig:3PPair}
\end{figure}

\begin{figure}[h]
    \centering
    \includegraphics[height=2.5in]{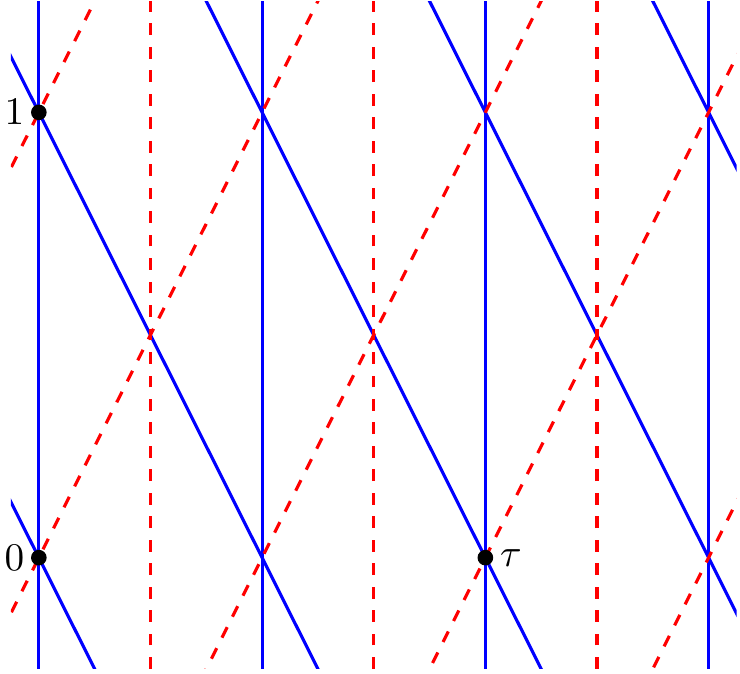}
    \caption{Vertical lifts of $N$ (blue) and $N'$ (red, dashed) for $j$ odd}
    \label{fig:3PPair_vert}
\end{figure}

The occurrence of the boundary slope $2/1$ for $N'$ is in stark contrast with the restricted set of boundary slopes available to other twist knot complements found in \cref{lem:SlopesOfTGS}.  The approach of the proof of \cref{lem:SlopesOfTGS} breaks down for the figure-8 knot, primarily because the minimal polynomial of $z$ is quadratic.  Moreover,

\begin{thm}
\label{thm:BoundSlopeQ}
Every number in $\BQ \cup \{\infty\}$ occurs as a boundary slope for some totally geodesic surface in the figure-8 knot complement.
\end{thm}

\begin{proof}
Let $a^p \ell^q$ be a parabolic element fixing $\infty$.  We can find $w a^m \ell^n w^{-1}$ such that the trace condition \cref{eq:TraceCondition} is satisfied.
As computed in the proof of \cref{lem:SlopesOfTGS}, we must satisfy
\begin{equation*}
16 nq + (2n+m)(2q+p) z^2 \in \BQ z
\end{equation*}
Since the minimal polynomial of $z$ is $z^2-z+1=0$, satisfying the above is equivalent to satisfying
\begin{equation*}
16 nq - (2n+m)(2q+p) = 0
\end{equation*}
For any pair $(p,q)$, there exists such a pair $(m,n)$.  Thus the group $\Delta$ generated by $a^p \ell^q$ and $w a^m \ell^n w^{-1}$ is a non-elementary Fuchsian subgroup.  Because the figure-8 knot is arithmetic, every non-elementary Fuchsian subgroup in $\pi_1(M)$ sits inside a finite coarea arithmetic Fuchsian subgroup of $\pi_1(M)$ \cite[Theorem 4]{MaclachlanReid87}. Thus there exists $\Delta' \supset \Delta$ such that $\Delta'$ is the stabilizer of a hyperplane which corresponds to a totally geodesic surface in the figure-8 knot complement admitting boundary slope $p/q$.
\end{proof}


\section{Computer experiments and questions}
\label{sec:ComputerExperimentsAndQuestions}

\subsection{Trace fields of twist knots}

The condition that the trace field $\bbQ(z_j)$ of $K_j$ for an odd prime $j$ has no proper real subfield is used frequently in the proof of \cref{thm:UniqueThricePuncturedSphere} and the lemmas leading up to it. Using Magma, we observe that for $1 < j < 100$, the trace field $\bbQ(z_j)$ has no proper subfield except for $\bbQ$, see \cite{Codes} for the code. The following question arises naturally:

\begin{quest}
\label{ques:NoProperSubfield}
Is $\bbQ$ the unique proper subfield of the trace field $\bbQ(z_j)$ of $K_j$?
\end{quest}

We also checked that the Galois group of the Riley polynomial $\Gal(\Lambda_j)$ is congruent to $S_j$, the full symmetric group on $j$ letters for $1< j \leq 99$, see \cite{Codes}. This implies that there are no proper subfields between $\bbQ(z_j)$ and $\bbQ$ since a subgroup of index $j$ in $S_j$ is necessarily maximal. Therefore, it is natural to ask: 

\begin{quest}
\label{quest:GaloisGroup}
Is $\Gal(\Lambda_j) \cong S_j$ for all $j$? 
\end{quest}

We conjecture that the answer is yes to both questions. By \cref{thm:UniqueThricePuncturedSphere}, an affirmative answer to either question implies that $N$ is the unique totally geodesic surface in $M_j$ all for odd $j$. 

Finally, using Magma we compute the class number of $\bbQ(z_j)$ and observe that it is $1$ for all $2 \leq  j \leq 19$, see \cite{Codes}. As a consequence of the argument in the proof of \cref{thm:UniqueThricePuncturedSphere}, the cusp points of $\Gamma_j$ form a proper subset of the trace field $\bbQ(z_j)$. In other words, for odd integer $2\leq j \leq 19$, the group $\Gamma_j$ gives an example of a Kleinian group whose trace field has class number $1$ and contains cusp points as a proper subset. Examples of this phenomenon are only known to exist for Kleinian groups whose trace field has class number strictly bigger than 1 \cite[Theorem 6.1]{LR}. One may ask if the class number of $\bbQ(z_j)$ is $1$ for all $j$. However, the class number of $\BQ(z_{20})$ is computed to be $2$, see \cite{Codes}. 

\subsection{Homological behavior of $N$ in congruence covers}

We study the homological behavior of totally geodesic surfaces in congruence covers of $\Gamma_j$. This study is motivated by the analogy between totally geodesic surfaces in (arithmetic) hyperbolic manifolds and generic hyperplane sections in the projective variety described in $\cite{B}$. For a (complex) projective variety $M$ of dimension $n$, the Lefschetz Hyperplane Theorem says that the map
\[ H_i(M \cap H) \to H_i(M)\]
is injective for all $i \geq n$ and for a generic hyperplane $H$. Replacing "hyperplane section" by "totally geodesic surface in hyperbolic $3$-manifold", one can ask similar questions about the behavior of the induced map between the first homology of totally geodesic surfaces and the ambient hyperbolic $3$-manifold. 

To describe our question, we need to set up some notations. Let $\mathcal{O}_j$ be the ring of integers of $\bbQ(z_j)$ and identify $\Gamma_j$ with its image in $\PSL_2(\mathcal{O}_j)$. Let $\Gamma_j(I) := \ker(\Gamma_j \to \PSL_2(\mathcal{O}_j/I))$ be the principal congruence cover of level $I$ for some ideal $I$ and $M_j(I)$ be the cover of $M_j$ that corresponds to $\Gamma_j(I)$. Let $N_1,\dots,N_m$ be the lifts of $N$ in $M_j(I)$. 

\begin{quest}
What are the possible behaviors of the map
\[H_1(N_1 \cup \dots \cup N_m;\bbZ) \to H_1(M_j(I);\bbZ)\]
as $I$ varies over the prime ideals in $\mathcal{O}_j$? 
\end{quest}

We first investigate this question when $I$ is a prime with residue degree $1$, i.e.~when $\mathcal{O}_j/ I \cong \mathbb{F}_p$, a finite field of prime order $p$. Suppose that $\mathfrak{p}$ is a prime ideal in $\mathcal{O}_j$ lying above an odd prime $p$ such that $\mathcal{O}_j/\mathfrak{p} \cong \mathbb{F}_p$. In \cite[Proposition 4]{Riley}, Riley showed that the representation $\Gamma_j \to \PSL_2(\mathbb{F}_p)$ coming from the reduction map is surjective for all odd prime $p$. For all odd prime $p$, $\Delta_j$ maps onto the group 
\[\left\langle \begin{pmatrix}
1 & 1 \\ 0 & 1 \\
\end{pmatrix}, \begin{pmatrix}
-1 & 0 \\ 4 & -1 \\
\end{pmatrix} \right\rangle = \PSL_2(\mathbb{F}_p).\] 
under the reduction map
Therefore for any prime ideal $\mathfrak{p}$ with residue degree $1$ lying above an odd prime $p$, $M_j(\mathfrak{p})$ contains exactly one unique lift of the totally geodesic thrice-puncture sphere, and this cover has fundamental group $\Delta_j(\mathfrak{p}) = \Delta \cap \Gamma_j(\mathfrak{p})$. For all examples that we computed with Magma, we observed that $H_1(\Delta_j(\mathfrak{p}))$ maps onto $H_1(\Gamma_j(\mathfrak{p}))$ under the induced map for prime ideals $\mathfrak{p}$ of residue degree $1$ lying above an odd prime $p$. In particular, we verified that $H_1(\Delta_j(\mathfrak{p}))$ surjects $H_1(\Gamma_j(\mathfrak{p}))$ for $3 \leq j \leq 20$ that has a prime ideal of residue degree $1$ lying above $3$ and $5$, see \cite{Codes}.

\begin{quest}
Is the induced map $H_1(\Delta_j(\mathfrak{p})) \to H_1(\Gamma_j(\mathfrak{p}))$ surjective for all prime ideals $\mathfrak{p}$ of residue degree 1 lying above odd prime $p$?  
\end{quest}


Finally, we record some observations of patterns in the first homology of these principal congruence covers in \cref{tab:HomologyOfPrincipalCongruenceCovers}. In the table, we denote the homology group 
\[\bbZ^r_0\oplus (\bbZ/a_1\bbZ)^{r_1} \oplus \dots \oplus (\bbZ/a_s\bbZ)^{r_s}\]
by $[0^{r_0},a_1^{r_1},\dots,a_s^{r_s}]$ where $a_1,\dots,a_s$ are the invariant factors. Furthermore, repetition of a prime $p$ indicates a prime with different prime ideals lying above of residue degree 1. See , see \cite{Codes} for the code used to produce the table.

\begin{table}[h!]
   \centering
   \begin{tabular}{|c|c|c|c|}
   \hline
    $p$ & $j$ & $k$ & $H_1(\Gamma_j(\mathfrak{p}))$   \\
   \hline
   3 &  $4k+1$ & $1 \leq k \leq 13$ & $[0^4,(3k+1)^2]$ \\
   \hline
   5 & $10k + 1$ & $1 \leq k \leq 4$ & $[0^{12},(6k+1)^5]$\\
    \hline
   5 & $10k + 3$ & $1 \leq k \leq 4$ & $[0^{17},(3k+1)^4,6k+2]$\\
    \hline
    \end{tabular}
    \caption{Homology of Principal Congruence Covers}
    \label{tab:HomologyOfPrincipalCongruenceCovers}
\end{table}

\subsection{Finding surfaces in the figure-8 knot complement}

Given the linear relation between $p/q$ and $m/n$ that is described in both \cref{lem:SlopesOfTGS} and the proof of \cref{thm:BoundSlopeQ}, we ask how restrictive such relations between boundary slopes must be:

\begin{quest}
Can any pair of numbers in $\BQ \cup \{\infty\}$ occur simultaneously as boundary slopes for some totally geodesic surface of $M$?
\end{quest}

We give some computed examples of totally geodesic surfaces $\Sigma$ in $M$ below. We identify these by specifying a vertical hyperplane in $\bbH^3$ and its stabilizer. These hyperplanes have the form $\bbH(s/t,p/q)$ where $0 \leq s/t < 1$ is the point of intersection with $H^2_{\BR}$ and $p/q$ is the slope of the vertical plane at infinity. We denote the stabilizers of these hyperplanes at infinity by $\Delta(s/t,p/q)$. 
 
\begin{example}
The plane $\bbH(0,1/0)$ is $\bbH^2_{\bbR}$ and has stabilizer 
\begin{equation*}
\Delta(0,1/0) 
= \langle a, {}^{awa^{-1}w}a \rangle 
= \left\langle \begin{pmatrix} 1 & 1 \\ 0 & 1 \end{pmatrix}, \begin{pmatrix} -1 & 1 \\ -4 & 3 \end{pmatrix} \right\rangle
\end{equation*}
This is the thrice-punctured sphere $N$.
\end{example}

\begin{example}
The plane $\bbH(0,2/1)$ has stabilizer
\begin{equation*}
\displaystyle\Delta(0,2/1) 
\supseteq \langle a^2\ell, {}^{ab^{-1}} a \rangle 
= \left\langle \begin{pmatrix} -1 & -4z \\ 0 & -1 \end{pmatrix}, \begin{pmatrix} 2 & -z \\ -z+1 & 0 \end{pmatrix} \right\rangle
\end{equation*}
Note that
\begin{equation*}
\begin{pmatrix} -1 & -4z \\ 0 & -1 \end{pmatrix} \begin{pmatrix} 2 & -z \\ -z+1 & 0 \end{pmatrix} = 
\begin{pmatrix}
2 & z\\ z-1 & 0
\end{pmatrix}
\end{equation*}
This is the thrice-punctured sphere $N'$.  
\end{example}

Both $N$ and $N'$ admit precisely two boundary slopes, but that is not necessarily true for all totally geodesic surfaces:


\begin{example}
The plane $\bbH(0,0)$ has stabilizer 
\begin{align*}
\Delta(0,0) 
& \supseteq \langle \ell,  {}^{w}a^6\ell, {}^{\ell w^{-1}}a^6\ell, {}^{\ell w^{-1}}\ell\rangle
\\
& \phantom{\supseteq} = \left\langle \begin{pmatrix} -1 & -4z + 2 \\ 0 & -1 \end{pmatrix}, \begin{pmatrix} -1 & 0 \\ -8z + 4 & -1\end{pmatrix}, 
\begin{pmatrix} 11 & -24z + 12 \\ -8z + 4 & -13  \end{pmatrix} \right\rangle
\end{align*}
The set of boundary slopes of $\Delta(0,0)$ contains but may not be limited to $\{0,6/1\}$.
\end{example}

\begin{example}
The plane $\bbH(0,1/1)$ has stabilizer
\begin{align*}
\Delta(0,1/1)
& \supseteq \langle a\ell, {}^{w}a^{10}\ell^3,  {}^{\ell v}a\ell^{-1} \rangle \\
& \phantom{\supseteq} = \left\langle \begin{pmatrix} -1 & -4z + 1 \\ 0 & -1 \end{pmatrix}, \begin{pmatrix} -1 & 0 \\ -16 z + 12 & -1 \end{pmatrix} \begin{pmatrix} 25 & -52z + 13 \\ -16z + 12 & -27\end{pmatrix} \right\rangle
\end{align*}
The set of boundary slopes of $\Delta(0,1/1)$ contains but may not be limited to $\{1/1,10/3,-1/1\}$. 
\end{example}

Given the existence of sets of boundary slopes that are not precisely two numbers in $\BQ \cup \{\infty\}$, we ask:

\begin{quest}
Does there exist a totally geodesic surface of $M$ admitting exactly $1$ boundary slope? arbitrarily many boundary slopes?
\end{quest}


\appendix

	\bigskip
	\bigskip
	
	\noindent\textsc{Department of Mathematics,
		Temple University, \\
		Wachman Hall Rm. 638 \\
		1805 N. Broad St \\
		Philadelphia PA, 19122 USA \\
		\emph{E-mail addresses:} {\bf khanh.q.le@temple.edu}, {\bf rebekah.palmer@temple.edu},}

\end{document}